\documentclass[12pt]{article}

\usepackage{graphicx}
\usepackage{mathptmx}

\usepackage{latexsym, epsfig,cite, psfrag,eepic,colordvi,bm}
\usepackage{graphics,color}
\usepackage{amsmath,amsfonts,amssymb,amscd}
\usepackage[all]{xy}

\usepackage{amsthm,amsmath,amssymb}

\usepackage{graphicx}

\usepackage[colorlinks=true,citecolor=black,linkcolor=black,urlcolor=blue]{hyperref}

\theoremstyle{plain}
\newtheorem{theorem}{Theorem}[section]
\newtheorem{lemma}[theorem]{Lemma}
\newtheorem{corollary}[theorem]{Corollary}
\newtheorem{proposition}[theorem]{Proposition}

\theoremstyle{definition}
\newtheorem{definition}[theorem]{Definition}
\newtheorem{example}[theorem]{Example}

\theoremstyle{remark}
\newtheorem{remark}[theorem]{Remark}

\date{}

\title{\bf Variation of the local topological structure of graph embeddings}

\author{Ricky X. F. Chen$^a$, Christian M. Reidys$^b$\\
\small Department of Mathematics and Computer Science\\[-0.8ex]
\small University of Southern Denmark, Campusvej 55\\[-0.8ex]
\small DK-5230, Odense M, Denmark\\
\small\tt $^a$chen.ricky1982@gmail.com,$^b$duck@santafe.edu
}

\begin{document}

\maketitle

\begin{abstract}
The $2$-cell embeddings of graphs on closed surfaces have been widely studied.
It is well known that ($2$-cell) embedding a given graph $G$ on a closed
orientable surface is equivalent to cyclically ordering
the edges incident to each vertex of $G$. In this paper, we study the following problem:
given a genus $g$ embedding $\mathbb{E}$ of the graph $G$, if we randomly rearrange the edges
around a vertex, i.e., re-embedding,
what is the probability of the resulting embedding $\mathbb{E}'$ having genus $g+\Delta g$?
We give a formula to compute this probability.
Meanwhile, some other known and unknown results are also obtained. For example,
we show that the probability of preserving the genus
is at least $\frac{2}{deg(v)+2}$ for re-embedding any vertex $v$ of degree $deg(v)$
in a one-face embedding; and we obtain a necessary condition for a given embedding of $G$ to
be an embedding with the minimum genus.

\bigskip\noindent \textbf{Keywords:} Graph embedding; Plane permutation; Genus; Hypermap

\noindent\small Mathematics Subject Classifications: 05C30; 05C10; 97K30
\end{abstract}

\section{Introduction}

Graph embedding is one of the most important topics in topological graph theory.
In particular, $2$-cell embeddings of graphs (loops and multiple edges allowed) have been widely studied.
A $2$-cell embedding of a given graph $G$ on a closed surface of genus $g$, $S_g$,
is an embedding on $S_g$ such that every face is homeomorphic to an open disk.
A $2$-cell embedding is also called a map.
The closed surfaces could be either orientable or unorientable. In this paper,
we restrict ourselves to orientable case.
Besides, by embedding we always mean $2$-cell embedding.
Note, by the classification theorem, any orientable
closed surface of genus $g$ is homeomorphic to the connected sum of $g$ tori.

There are many interesting topics on graph embedding in the literature. For instance,
given a graph $G$, what is the minimum (resp., maximum) genus $g$ such that there exists a $2$-cell embedding
of $G$ on $S_g$? For studies in detail, we refer the readers
to~\cite{furst,nebe,xuong1,xuong2,jung,white,liu,liu2,youngs,thoma}
and references therein.
Let $g_{min}(G)$ and $g_{max}(G)$ denote the minimum and the maximum
genus $g$ of the embeddings of $G$, respectively.
In Duke~\cite{duke}, an ``interpolation" theorem showed that for any $g_{min}(G)\leq g \leq g_{max}(G)$,
there exists an embedding of $G$ on $S_g$.
Assume $G$ has $e$ edges and $v$ vertices, and embedded on $S_g$ via the embedding $\mathbb{E}$.
The number $\beta(G)=e-v+1$
is called the betti number of $G$. According to Euler's characteristic formula, there holds
\begin{align}
v-e+f=2-2g \quad \Longleftrightarrow \quad 2g=\beta(G)+1-f,
\end{align}
where $f\geq 1$ is the number of faces of $\mathbb{E}$. Thus,
the largest possible value of $g$ is $\lfloor\frac{\beta(G)}{2}\rfloor$.
If $g_{max}(G)=\lfloor\frac{\beta(G)}{2}\rfloor$, $G$ is called upper
embeddable. When is $G$ upper embeddable? See studies in~{\cite{furst,nebe,xuong1,xuong2,jung,white,liu,liu2}}.

It is well known that an embedding of $G$ on a closed orientable surface can
be equivalently represented by $G$ with a specified cyclic order of edges around (i.e., incident to)
each vertex of $G$, i.e., the topological structure of the embedding is implied
in these cyclic orderings of edges~\cite{edmonds,youngs}. Any variation of the local topological structure around a vertex,
i.e., the cyclic order of edges around the vertex, may change the topological properties of
the embedding, e.g., the genus of the embedding.

Plane permutations were recently used to study hypermaps in Chen and Reidys~\cite{chr-1}.
It proved to be quite effective to enumerate hypermaps with one face.
Besides, plane permutations allow to study the transposition and
block-interchange distance of permutations as well as the reversal
distance of signed permutations in a unified simple framework~\cite{chr-2}.
Since maps are specific class of hypermaps,
it is natural to study graph embeddings using plane permutations as well.

This paper is organized as follows: in Section $2$,
we recall some basics of plane permutations~\cite{chr-1} for later use.
In Section $3$, we study embeddings with one face. These objects have been studied in many
fields~\cite{chap,walsh1,lzv,hz,zv,IJ,jac,pen,reidys} where the enumeration aspect
is the main interest. Our interest in this paper is to understand the following problems:
assume there exists a one-face embedding $\mathbb{E}$ for the graph $G$.
How many different ways are there of changing the local embedding (re-embedding) around a vertex
without changing the genus? By changing the local embedding, we mean changing
the cyclic order of edges around the vertex. For a given vertex, is there another local embedding around it to
preserve the genus? As results, we show that the probability of preserving the genus
is at least $\frac{2}{deg(v)+2}$ for re-embedding any vertex $v$ of degree (i.e., valence) $deg(v)$.
Also, there is at least one alternative way to re-embed a vertex $v$ preserving the genus if
$deg(v)\geq 4$. In Section~$4$, in order to study embeddings with more than
one face, we generalize plane permutations into $k$-cyc plane permutations.
We study more general questions, e.g., given an embedding $\mathbb{E}$ for the graph $G$,
what is the maximum (resp., minimum) genus can be achieved by changing the local embedding of one of the
vertices of $G$? For a vertex with larger degree, there are more
alternatives to rearrange the edges around it.
Is it true that re-embedding a vertex with larger degree always achieve a higher (resp., lower) genus than
re-embedding a vertex with smaller degree? and so on.
As results, we obtain a local version of the interpolation theorem which also provides an
easy approach to roughly estimate the range $[g_{min}(G), g_{max}(G)]$, as well as
a necessary condition for an embedding of $G$ to be an embedding with the minimum genus.


\section{Plane permutations}


Let $\mathcal{S}_n$ denote the group of permutations, i.e.~the group
of bijections from $[n]=\{1,\dots,n\}$ to $[n]$, where the multiplication
is the composition of maps.
We shall discuss the following three representations of a permutation $\pi$:\\
\emph{two-line form:} the top line lists all elements in $[n]$, following the natural order.
The bottom line lists the corresponding images of elements on the top line, i.e.
\begin{eqnarray*}
\pi=\left(\begin{array}{ccccccc}
1&2& 3&\cdots &n-2&{n-1}&n\\
\pi(1)&\pi(2)&\pi(3)&\cdots &\pi({n-2}) &\pi({n-1})&\pi(n)
\end{array}\right).
\end{eqnarray*}
\emph{cycle form:} regarding $\langle \pi\rangle$ as a cyclic group, we represent $\pi$ by its
collection of orbits (cycles).
The set consisting of the lengths of these disjoint cycles is called the cycle-type of $\pi$.
We can encode this set into a non-increasing integer sequence $\lambda=\lambda_1
\lambda_2\cdots$, where $\sum_i \lambda_i=n$, or as $1^{a_1}2^{a_2}\cdots n^{a_n}$, where
we have $a_i$ cycles of length $i$.
A cycle of length $k$ will be called a $k$-cycle. A cycle of odd and even length will be called an odd
and even cycle, respectively. It is well known that all permutations of a same cycle-type
forms a conjugacy class of $\mathcal{S}_n$.

\begin{definition}[Plane permutation]
A plane permutation on $[n]$ is a pair $\mathfrak{p}=(s,\pi)$ where $s=(s_i)_{i=0}^{n-1}$
is an $n$-cycle and $\pi$ is an arbitrary permutation on $[n]$. The permutation $D_{\mathfrak{p}}=
s\circ \pi^{-1}$ is called the diagonal of $\mathfrak{p}$.
\end{definition}\label{2def1}

Given $s=(s_0s_1\cdots s_{n-1})$,
a plane permutation $\mathfrak{p}=(s,\pi)$ can be represented by two aligned rows:
\begin{equation}
(s,\pi)=\left(\begin{array}{ccccc}
s_0&s_1&\cdots &s_{n-2}&s_{n-1}\\
\pi(s_0)&\pi(s_1)&\cdots &\pi(s_{n-2}) &\pi(s_{n-1})
\end{array}\right)
\end{equation}
Indeed, $D_{\mathfrak{p}}$ is determined by the diagonal-pairs (cyclically) in the two-line
representation here, i.e., $D_{\mathfrak{p}}(\pi(s_{i-1}))=s_i$ for $0<i< n$, and
$D_{\mathfrak{p}}(\pi(s_{n-1}))=s_0$.
For convenience, we always assume $s_0=1$ in the following and
we mean by ``the cycles of $\mathfrak{p}=(s,\pi)$'' the cycles of $\pi$.

Given a plane permutation $\mathfrak{p}=(s,\pi)$ on $[n]$ and a sequence $h=h_1h_2\cdots h_{n-1}$ on $[n-1]$,
let $s^h=(s_0, s_{h_1}, s_{h_2}, \ldots s_{h_{n-1}})$ and $\pi^h=D_{\mathfrak{p}} \circ s^h$.
We write $(s^h, \pi^h)=\chi_h \circ (s,\pi)$.
In particular, if $h=12\cdots (i-1) \underline{(j+1)\cdots l},  \underline{i \cdots j} (l+1)\cdots (n-1)$
where $0< i\leq j <l <n$,
we have
$$
s^h=(s_0,s_1,\dots,s_{i-1},\underline{s_{j+1},\dots,s_l},
\underline{s_i,\dots,s_j},s_{l+1},\dots,s_{n-1}),
$$
i.e.~the $n$-cycle obtained by transposing the blocks $[s_i,s_j]$ and $[s_{j+1},s_l]$.
Then, $(s^h, \pi^h)$ can be represented as
\begin{eqnarray*}
\left(
\vcenter{\xymatrix@C=0pc@R=1pc{
\cdots s_{i-1}\ar@{->>}[d]  & s_{j+1}\ar@{--}[dl] &\cdots & s_{l-1}& s_l\ar@{--}[dl]\ar@{->>}[d] &
 s_i\ar@{--}[dl] &\cdots & s_{j-1} & s_{j}\ar@{--}[dl]\ar@{->>}[d] & s_{l+1}  \cdots\\
\cdots \pi(s_{j}) & \pi(s_{j+1}) & \cdots  & \pi(s_{l-1}) & \pi(s_{i-1}) & \pi(s_i) &\cdots & \pi(s_{j-1})&
\pi(s_l)& \pi(s_{l+1}) \cdots
}}
\right)
\end{eqnarray*}
Note that the bottom row of the two-row representation of $(s^h,\pi^h)$
is obtained by transposing the blocks $[\pi(s_{i-1}),\pi(s_{j-1})]$ and
$[\pi(s_{j}),\pi(s_{l-1})]$ of the bottom row of $(s,\pi)$. In this particular case,
we denote the sequence $h$ as $(i,j,j+1,l)$ for short and refer to $\chi_h$ a transpose.
For general $h$, we observe that the two-row form of $(s^h,\pi^h)$ is obtained by
rearranging the diagonal-pairs of $(s,\pi)$.
As a result, we observe

\begin{lemma}\cite{chr-1}\label{2lem1}
Given a plane permutation $(s,\pi)$ on $[n]$ and a transpose $\chi_h$, where $h=(i,j,j+1,l)$ and
$0< i\leq j <l <n$. Then $\pi(s_r)=\pi^h(s_r)$ if $r\in \{0,1,\ldots,n-1\}\setminus \{i-1,j,l\}$, and
$$
\pi^h(s_{i-1})=\pi(s_{j}), \quad \pi^h(s_j)=\pi(s_l), \quad \pi^h(s_l)=\pi(s_{i-1}).
$$
\end{lemma}

We shall proceed by analyzing the induced changes of the $\pi$-cycles when passing
to $\pi^h$. By Lemma~\ref{2lem1}, only the $\pi$-cycles containing $s_{i-1}$, $s_{j}$,
$s_l$ will be affected.

\begin{lemma}\cite{chr-1}\label{2lem2}
Given $(s,\pi)$ and a transpose $\chi_h$ where $h=(i,j,j+1,l)$ and $0< i \leq j < l <n$,
then there exist the following six scenarios for the pairs $(\pi,\pi^h)$:
\begin{center}
\begin{tabular}{|c|c|c|}
\hline
Case~$1$ &$\pi$& $(s_{i-1},v_1^i,\ldots v_{m_i}^i),  (s_j, v_1^j,\ldots v_{m_j}^j), (s_l,v_1^l,\ldots v_{m_l}^l)$\\\cline{2-3}
&$\pi^h$&$ (s_{i-1},v_1^j,\ldots v_{m_j}^j, s_{j}, v_1^l,\ldots v_{m_l}^l, s_l, v_1^i, \ldots v_{m_i}^i)$\\
\hline\hline
Case~$2$ &$\pi$& $(s_{i-1},v_1^i,\ldots v_{m_i}^i, s_l, v_1^l,\ldots v_{m_l}^l, s_j, v_1^j, \ldots v_{m_j}^j)$\\\cline{2-3}
&$\pi^h$& $(s_{i-1}, v_1^j,\ldots v_{m_j}^j),  (s_j, v_1^l,\ldots v_{m_l}^l), (s_l, v_1^i, \ldots v_{m_i}^i)$\\
\hline\hline
Case~$3$ &$\pi$& $(s_{i-1},v_1^i,\ldots v_{m_i}^i, s_j, v_1^j, \ldots v_{m_j}^j, s_l, v_1^l,\ldots v_{m_l}^l)$\\\cline{2-3}
&$\pi^h$& $(s_{i-1},v_1^j,\ldots v_{m_j}^j, s_l, v_1^i,\ldots v_{m_i}^i, s_j, v_1^l, \ldots v_{m_l}^l)$\\
\hline\hline
Case~$4$ &$\pi$& $(s_{i-1}, v_1^i, \ldots v_{m_i}^i, s_{j}, v_1^j, \ldots v_{m_j}^j), (s_l, v_1^l, \ldots v_{m_l}^l)$\\\cline{2-3}
&$\pi^h$&$ (s_{i-1}, v_1^j, \ldots v_{m_j}^j), (s_{j}, v_1^l, \ldots v_{m_l}^l, s_l, v_1^i, \ldots v_{m_i}^i)$\\
\hline\hline
Case~$5$ &$\pi$& $(s_{i-1}, v_1^i, \ldots v_{m_i}^i),  (s_{j}, v_1^j, \ldots v_{m_j}^j, s_l, v_1^l, \ldots v_{m_l}^l)$\\\cline{2-3}
&$\pi^h$&$ (s_{i-1}, v_1^{j},\ldots v_{m_j}^j, s_l, v_1^i, \ldots v_{m_i}^i), (s_{j}, v_1^l, \ldots v_{m_l}^l)$\\
\hline\hline
Case~$6$ &$\pi$& $(s_{i-1}, v_1^i, \ldots v_{m_i}^i, s_l, v_1^l, \ldots v_{m_l}^l), (s_{j}, v_1^j, \ldots v_{m_j}^j)$\\\cline{2-3}
&$\pi^h$&$ (s_{i-1}, v_1^j, \ldots v_{m_j}^j, s_{j}, v_1^l, \ldots v_{m_l}^l), (s_l, v_1^i, \ldots v_{m_i}^i)$\\
\hline
\end{tabular}
\end{center}
\end{lemma}
\begin{proof} We shall only prove Case~$1$ and Case~$2$, the remaining four cases
can be shown analogously.
For Case~$1$, the $\pi$-cycles containing $s_{i-1}$,~$s_j$,~$s_l$ are
$$
(s_{i-1},v_1^i,\ldots v_{m_i}^i),  (s_j, v_1^j,\ldots v_{m_j}^j), (s_l,v_1^l,\ldots v_{m_l}^l).
$$
Lemma~\ref{2lem1} allows us to identify the new cycle structure by inspecting the critical
points $s_{i-1}$, $s_j$ and $s_l$.
Here we observe that all three cycles merge and form a single $\pi^h$-cycle
\begin{eqnarray*}
(s_{i-1}, \pi^h(s_{i-1}),(\pi^h)^2(s_{i-1}),\ldots)&=&(s_{i-1}, \pi(s_j), \pi^2(s_j),\ldots )\\
&=& (s_{i-1},v_1^j,\ldots v_{m_j}^j, s_{j}, v_1^l,\ldots v_{m_l}^l, s_l, v_1^i, \ldots v_{m_i}^i).
\end{eqnarray*}
For Case $2$, the $\pi$-cycle containing $s_{i-1}$,~$s_j$,~$s_l$ is
$$
(s_{i-1},v_1^i,\ldots v_{m_i}^i, s_l, v_1^l,\ldots v_{m_l}^l, s_j, v_1^j, \ldots v_{m_j}^j).
$$
We compute the $\pi^h$-cycles containing $s_{i-1}$, $s_j$ and $s_l$ in $\pi^h$ as
\begin{eqnarray*}
(s_{i-1},\pi^h(s_{i-1}),(\pi^h)^2(s_{i-1}),\ldots)&=&(s_{i-1},\pi(s_{j}),
\pi^2(s_{j}),\ldots)=(s_{i-1}, v_1^j,\ldots v_{m_j}^j)\\
(s_j,\pi^h(s_j),(\pi^h)^2(s_j),\ldots)&=&(s_j,\pi(s_{l}),\pi^2(s_{l}),\ldots)=(s_j, v_1^l,\ldots v_{m_l}^l)\\
(s_l,\pi^h(s_l),(\pi^h)^2(s_l),\ldots)&=&(s_l,\pi(s_{i-1}),\pi^2(s_{i-1}),\ldots)=(s_l, v_1^i,\ldots v_{m_i}^i)
\end{eqnarray*}
whence the lemma.

\end{proof}


\begin{definition}
Two plane permutations $(s,\pi)$ and $(s',\pi')$ are equivalent if there exists a
permutation $\alpha$ such that
$$
s=\alpha s' \alpha^{-1}, \quad \pi=\alpha\pi'\alpha^{-1}, \quad \alpha(1)=1.
$$
\end{definition}

For two equivalent plane permutations
$\mathfrak{p}=(s,\pi)$ and $\mathfrak{p}'=(s',\pi')$, we have $s=s'$ if and only if
$\pi=\pi'$. Clearly, the equation $\alpha s'\alpha^{-1}=s=s'$ restricts $\alpha$ to be a shift
within the $n$-cycle $s'$ and the latter has to be trivial due to $\alpha(1)=1$.

Let $U_{D}$ denote the set of plane permutations having
$D$ as diagonals for some fixed permutation $D$ on $[n]$.
Note $\mathfrak{p}=(s,\pi)\in U_D$ iff $D=D_{\mathfrak{p}}=s\circ \pi^{-1}$.
Then, the number $|U_D|$ enumerates the ways to write $D$ as a
product of an $n$-cycle with another permutation. Or equivalently, assuming
$D$ is of cycle-type $\lambda$, in view of
$$
D=s\pi^{-1} \quad \Longleftrightarrow \quad (12\cdots n)=\gamma s\gamma^{-1}=
(\gamma D \gamma^{-1})( \gamma \pi \gamma^{-1}),
$$
where $\gamma$ is unique if $\gamma(1)=1$, $|U_D|$ is also the number
of factorizations of $(12\cdots n)$ into a permutation of cycle-type $\lambda$ and another permutation,
i.e., rooted hypermaps having one face.
A rooted hypermap is a triple of permutations $(\alpha,\beta_1,\beta_2)$,
such that $\alpha=\beta_1\beta_2$. The cycles in $\alpha$ are called faces, the cycles in $\beta_1$
are called (hyper)edges, and the cycles in $\beta_2$ are called vertices. If $\beta_1$ is
an involution without fixed points, the rooted hypermap is an ordinary rooted map.
We refer to~\cite{chap,walsh1,lzv,hz,zv,IJ,jac,pen,reidys,chr-3} and references therein for an
in-depth study of hypermaps and maps.

Let $\mu,\eta$ be partitions of $n$. We write $\mu\rhd_{2i+1}\eta$ if $\mu$
can be obtained by splitting one $\eta$-block into $(2i+1)$ non-zero parts. Let
furthermore $\kappa_{\mu,\eta}$ denote the number of different ways to obtain
$\eta$ from $\mu$ by merging $\ell(\mu)-\ell(\eta)+1$ $\mu$-blocks into one, where
$\ell(\mu)$ and $\ell(\eta)$ denote the number of blocks in the partitions $\mu$ and $\eta$,
respectively.

Let $U_{\lambda}^{\eta}$ denote the set of plane permutations, $\mathfrak{p}=(s,\pi)\in U_D$,
where $D$ has cycle-type $\lambda$ and $\pi$ has cycle-type $\eta$.

\begin{theorem}\cite{chr-1}\label{2thm1}
Let $f_{\eta,\lambda}(n)=\vert U_{\lambda}^{\eta} \vert$.
Then, we have
\begin{equation}
f_{\eta,\lambda}(n)=\frac{\sum_{i=1}^{\lfloor\frac{n-\ell(\eta)}{2}\rfloor}\sum_{\mu\rhd_{2i+1}\eta}
\kappa_{\mu,\eta}f_{\mu,\lambda}(n)
+\sum_{i=1}^{\lfloor\frac{n-\ell(\lambda)}{2}\rfloor}\sum_{\mu\rhd_{2i+1}\lambda}
\kappa_{\mu,\lambda}f_{\mu,\eta}(n)}{n+1-\ell(\eta)-\ell(\lambda)}.
\end{equation}
\end{theorem}

\begin{corollary}\cite{chr-1}\label{2cor1}
Let $p_{k}^{\lambda}(n)$ denote the number of $\mathfrak{p}\in U_{D}$ having
$k$ cycles, where $D$ is of cycle-type $\lambda$.
\begin{equation}
p_k^{\lambda}(n)=\frac{\sum_{i=1}^{\lfloor\frac{n-k}{2}\rfloor}{k+2i\choose k-1}p_{k+2i}^{\lambda}(n)
+\sum_{i=1}^{\lfloor\frac{n-\ell(\lambda)}{2}\rfloor}\sum_{\mu\rhd_{2i+1}\lambda}
\kappa_{\mu,\lambda}p_k^{\mu}(n)}{n+1-k-\ell(\lambda)}.
\end{equation}
\end{corollary}

\begin{proposition}\cite{chr-1,zag}\label{2pro1}
\begin{align}
\arg\max_k\{p_k^{\lambda}(n)\neq 0\}=n+1-\ell (\lambda).
\end{align}
\end{proposition}

\section{Local variation of embeddings with one face}

As already mentioned, an embedding of the given graph $G$ can be combinatorially encoded into
$G$ with a specified cyclic order of edges incident to each vertex of $G$.
Such a cyclic ordering is also called a rotation system.
A graph with a rotation system on it is called a fatgraph.

Conventionally, a fatgraph of $n$ edges is encoded into a triple of permutations $(\alpha,\beta,\gamma)$ on $[2n]$.
This can be obtained as follows: given a fatgraph $F$, we firstly call the two ends of an edge as two half edges.
Label all half edges using the labels from the set $[2n]$ so that each label appears exactly once.
Then we immediately obtain two permutations $\alpha$ and $\beta$
where $\alpha$ is an involution without fixed points and each cycle consisting of the labels of
the two half edges of a same edge and each cycle in $\beta$ is the counterclockwise cyclic arrangement
of all half edges incident to a same vertex. The third permutation $\gamma=\alpha\beta$, which can be
interpreted as the set of counterclockwise boundaries of the fatgraph.
A boundary of the fatgraph is obtained as follows:
start from some half edge, and every time when we meet a half edge we next go to the half edge paired with the
counterclockwise neighbor of the present half edge until we meet the starting half edge again, the obtained cycle
is a boundary of the fatgraph which corresponds to a cycle in $\gamma$.
Starting from one half edge which does not appear in the former obtained
boundary (or boundaries) and continuing the traveling process, we obtain all the boundaries of the fatgraph.
If $\gamma$ has $k$ cycles, the fatgraph has $k$ boundaries, i.e., the corresponding embedding has $k$ faces.
Obviously, a different triple of permutations can be obtained by relabeling the half edges of the fatgraph.

For a given fatgraph $(\alpha, \beta, \gamma)$, it is well known that
$(\alpha, \alpha\beta, \alpha\gamma)=(\alpha, \gamma, \beta)$ is its Poincar\'{e} dual
which transforms a face into a vertex and vice versa.

In this section, we will focus on embeddings with one face, i.e., in the triple
$(\alpha,\beta, \gamma)$, $\gamma$ has only one cycle.
These maps are also called unicellular maps~\cite{chap,reidys}.
At this point, we observe:\\
{\bf Observation:} A unicellular maps $(\alpha,\beta,\gamma)$ can be encoded into a plane permutation $(s,\pi)$
as well, i.e., $s=\gamma$ and $\pi=\beta$.\\
See Figure~\ref{fig-fig1} for an example of a fatgraph
with one boundary. The two drawings there are the same unicellular map.
However, in the drawing on the righthand side the edges are drawn as ribbons so that
it is sometimes called ribbon graphs.
The corresponding plane permutation is
\begin{eqnarray*}
\mathfrak{p}=\left(\begin{array}{cccccccc}
1&2&3&4&5&6&7&8\\
1&6&7&8&3&4&5&2
\end{array}\right).
\end{eqnarray*}
These objects have been considered in many different contexts, e.g., the computation of
matrix integral \cite{zv}, moduli space of curves \cite{hz}, factorization of permutations \cite{IJ,ag,jac},
topological RNA and protein structure
\cite{reidys,pen}, etc.

\begin{figure}[!htb]
\centering
\includegraphics[scale=.6]{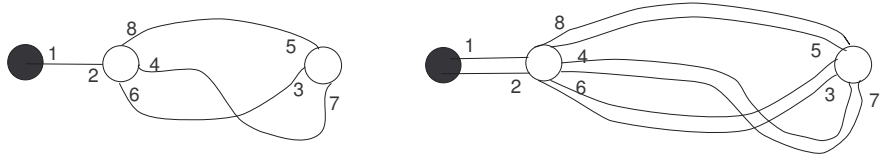}
{\centering \caption{A unicellular map with $4$ edges.}\label{fig-fig1}}

\end{figure}

In the following, we study embeddings (fatgraphs) in the framework of plane permutations.
Let $\mathfrak{p}=(s,\pi)$ encode a one-face embedding of the graph $G$.
The elements in the set on which the plane permutation is defined are
called half edges. A cycle of $\pi$ is also called a vertex and
$s$ is also called the boundary (i.e., face).
The genus of the plane permutation refers to the genus of the corresponding unicellular maps.
We focus on the local structure of unicellular maps,
which is motivated by Case~$3$ transpose in Lemma~\ref{2lem2} as follows:
given a plane permutation, if we apply a Case~$3$ transpose,
the set of half edges in each cycle of $\pi$ is not changed, which means the underlying graphs before and after the
transpose are the same. Namely, we may obtain different unicellular maps by rearranging the half edges around the vertices of a
given unicellular map. Thus we could ask,
given a unicellular map, if we randomly rearrange half edges around a vertex, what is the probability of
the obtained fatgraph is still a unicellular map?
For a given vertex in a unicellular map, whether it is the unique way of arranging all edges incident to the
vertex to achieve the genus in the present map? and so on.

\subsection{Variation of the embedding around one vertex}

At first, we consider the case where only half edges around one vertex are rearranged.
Let $\mathfrak{p}=(s,\pi)$ correspond to a one-face embedding of the graph $G$.
A vertex $v$ of $G$ can be represented as a cycle of $\pi$, which can be also naturally encoded into
a plane permutation
\begin{eqnarray*}
v=(s_v,\pi_v)=\left(\begin{array}{cccccccc}
s_{i_0}&s_{i_1}&s_{i_2}&s_{i_3}&\cdots&s_{i_{k-1}}\\
\pi(s_{i_0})&\pi(s_{i_1})&\pi(s_{i_2})&\pi(s_{i_3})&\cdots&\pi(s_{i_{k-1}})
\end{array}\right),
\end{eqnarray*}
where $s_v$ is a subsequence of $s$ consisting of the half edges around $v$ and $\pi_v$ is equal to $\pi$
with restriction to the half edges around $v$, i.e., the set $H(v)$.

Let $X$ denote the set of $k$-cycles $\theta$ on $H(v)$ such that the resulting embedding
has one face after rearranging half edges around $v$ according to $\theta$, $Y$
denote the set of sequences $h$ on $[k-1]$ such that $\chi_h \circ (s_v, \pi_v)$ has only one cycle.

\begin{theorem}\label{3thm1}
$|X|=|Y|$.
\end{theorem}
\begin{proof}
Let
\begin{eqnarray*}
\mathfrak{p}=\left(
\vcenter{\xymatrix@C=0pc@R=1pc{
\cdots s_{i_0} & s_{i_0+1} &{\cdots} & s_{i_1}\ar@{--}[dl] &{\quad\cdots} & s_{i_2}\ar@{-}[dl]& \cdots & s_{i_{k-2}} &{\cdots} & s_{i_{k-1}} \ar@{--}[dl] \cdots\\
\cdots \pi(s_{i_0})\ar@{--}[ur] & {\cdots} & \pi(s_{i_1-1}) & \pi(s_{i_1})\ar@{-}[ur] & {\cdots}  & \pi(s_{i_2}) & \cdots & \pi(s_{i_{k-2}})\ar@{--}[ur] & {\cdots} & \pi(s_{i_{k-1}}) \cdots
}}
\right).
\end{eqnarray*}
$Y \rightarrow X$: each $h=h_1h_2\cdots h_{k-1}\in Y$ uniquely induces a rearrangement of
the following diagonal blocks of $\mathfrak{p}$:
  $$
  [s_{i_0+1},s_{i_1}],\quad [s_{i_1+1},s_{i_2}],\quad \ldots \quad [s_{i_{k-2}+1},s_{i_{k-1}}],
  $$
where each segment (i.e., interval), e.g., $[s_{i_0+1},s_{i_1}]$, refers to the diagonal
block with the segment as the top row (or called top boundary).
The resulting plane permutation is still an embedding with one face.
Note in this operation, similar to transposes in Lemma~\ref{2lem2},
 we only change the image of the elements in the set $H(v)$ into
the elements in $H(v)$, all other cycles of $\mathfrak{p}$ are not changed.
If after the rearrangement of the diagonal blocks according to $h$,
the permutation on $H(v)$ forms only one cycle $\theta$,
then the resulting unicellular map has the same underlying graph $G$.
Namely, the resulting embedding with one face is obtained by rearranging the half edges around $v$
according to $\theta$. By construction, we have
$$
\theta(s_{i_0})=\pi(s_{i_{h_1-1}}),\ldots \theta(s_{i_{h_j}})=\pi(s_{i_{h_{j+1}-1}}),
\ldots \theta(s_{i_{h_{k-1}}})=\pi(s_{i_{k-1}}).
$$
\\
$X \rightarrow Y$: given $\theta \in X$, since the resulting embedding after rearrangement according to $\theta$
is still one-face embedding, then the corresponding plane permutation $(s', \pi')$ must have the form
\begin{eqnarray*}
\left(
\vcenter{\xymatrix@C=0pc@R=1pc{
\cdots s'_{i_0} & s'_{i_0+1} &{\cdots} & s'_{i_1}\ar@{--}[dl] &{\quad\cdots} & s'_{i_2}\ar@{-}[dl]& \cdots & s'_{i_{k-2}} &{\cdots} & s'_{i_{k-1}}\quad \ar@{--}[dl] \cdots\\
\cdots \pi'(s'_{i_0})\ar@{--}[ur] & {\cdots} & \pi'(s'_{i_1-1}) & \pi'(s'_{i_1})\ar@{-}[ur] & {\cdots}  & \pi'(s'_{i_2}) & \cdots & \pi'(s'_{i_{k-2}})\ar@{--}[ur] & {\quad \cdots \quad} & \pi'(s'_{i_{k-1}}) \cdots
}}
\right)
\end{eqnarray*}
where we assume $s'_0=s_0$. Since by construction the local structures are not changed except for
around $v$, $s_j=s'_j$ for $0 \leq j \leq i_0$. Assume
$$
H(v)=\{s'_{i_0},s'_{i_1},\ldots s'_{i_{k-1}}\}=\{\pi'(s'_{i_0}),\pi'(s'_{i_1}),\ldots \pi'(s'_{i_{k-1}})\}.
$$
Then, we have $\pi'(s'_{i_j})=\theta(s'_{i_j})$ for $0\leq j \leq k-1$.
It suffices to show that each diagonal block $[s'_{i_j+1},s'_{i_{j+1}}]$ for some $j$ is the same as
the diagonal block $[s_{i_l+1},s_{i_{l+1}}]$ for some $l$, i.e.,
\begin{eqnarray*}
\left(
\vcenter{\xymatrix@C=0pc@R=1pc{
s'_{i_j} & s'_{i_j+1}\ar@{--}[dl] & s'_{i_j+2} &\cdots & s'_{i_{j+1}}\ar@{--}[dl]\\
\pi'(s'_{i_j}) & \pi'(s'_{i_j+1})  &{\cdots} & \pi'(s'_{i_{j+1}-1}) &\pi'(s'_{i_{j+1}})
}}
\right)=
\left(
\vcenter{\xymatrix@C=0pc@R=1pc{
s_{i_l} & s_{i_l+1}\ar@{--}[dl] & s_{i_l+2} &\cdots & s_{i_{l+1}}\ar@{--}[dl]\\
\pi(s_{i_l}) & \pi(s_{i_l+1})  &{\cdots} & \pi(s_{i_{l+1}-1}) &\pi(s_{i_{l+1}})
}}
\right)
\end{eqnarray*}
{\bf Claim.} If $\pi'(s'_{i_j})=\pi(s_{i_l})$, the diagonal block $[s'_{i_j+1}, s'_{i_{j+1}}]$
and the diagonal block $[ s_{i_l+1}, s_{i_{l+1}}]$ are equal.\\
Note, if $\pi'(s'_{i_j})=\pi(s_{i_l})$, then
$$
s'_{i_j+1}=D_{\mathfrak{p}}\circ \pi'(s'_{i_j})=D_{\mathfrak{p}}\circ \pi(s_{i_l})=s_{i_l+1}.
$$
Since $s'_{i_j+1}$ is not in $H(v)$, $\pi'(s'_{i_j+1})=\pi(s'_{i_j+1})=\pi(s_{i_l+1})$.
Continuing the analysis, we have
$$
s'_{i_j+2}=D_{\mathfrak{p}}\circ \pi'(s'_{i_j+1})=D_{\mathfrak{p}}\circ \pi(s_{i_l+1})=s_{i_l+2},
$$
and so on, finally we come to $s'_{i_{j+1}}=s_{i_{l+1}}$. This affirms the claim.\\
Therefore, the sequence of the diagonal blocks $[s'_{i_j+1}, s'_{i_{j+1}}]$ are rearrangement of
the sequence of the diagonal blocks $[s_{i_l+1}, s_{i_{l+1}}]$. It is obvious that each
rearrangement of the diagonal blocks $[s_{i_l+1}, s_{i_{l+1}}]$ uniquely induces a rearrangement
of the diagonal-pairs of $(s_v,\pi_v)$ according to $h\in [k-1]$. This completes the proof.
\end{proof}

From the discussion above, we observe that given any vertex $v$, the half edges in $H(v)$ will
segment the plane permutation into $|H(v)|$ diagonal blocks. Each diagonal block
is completely determined by its left lower corner $\pi(s_{i_t})$ and its right upper
corner $s_{i_{t+1}}=s_v(s_{i_t})$. We can view these diagonal
blocks as arranged in a circular manner, as shown in Figure~\ref{fig-circ}.
To rearrange the half edges around $v$ is to rearrange these diagonal blocks circularly.

\begin{figure}[!htb]
\centering
\includegraphics[scale=.6]{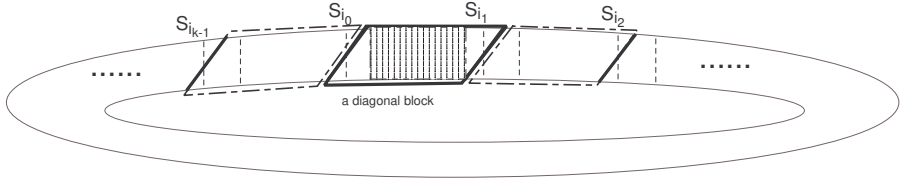}
{\centering \caption{Circular arrangement of diagonal blocks determined by the vertex $v$}\label{fig-circ}}

\end{figure}

Given a plane permutation $\mathfrak{p}=(s,\pi)$ with only one cycle, how many different $h$ such that
$\chi_h\circ (s,\pi)$ has only one cycle?
From the discussion in Section~$2$, we know that it is equivalent to factorizing $D_v$.
Let $R_v=|X|=|Y|$. Then we have

\begin{corollary}\label{3cor1}
Let $\mathfrak{p}=(s,\pi)$ correspond to a one-face embedding of the graph $G$,
$v$ is a vertex of $G$ and
\begin{eqnarray*}
v=(s_v,\pi_v)=\left(\begin{array}{cccccccc}
s_{i_0}&s_{i_1}&s_{i_2}&s_{i_3}&\cdots&s_{i_{k-1}}\\
\pi(s_{i_0})&\pi(s_{i_1})&\pi(s_{i_2})&\pi(s_{i_3})&\cdots&\pi(s_{i_{k-1}})
\end{array}\right).
\end{eqnarray*}
Assume $D_v=s_v\circ \pi_v^{-1}$ is of cycle-type $\lambda$. Then, we have
\begin{align}\label{3eq5}
R_v=p_1^{\lambda}(k).
\end{align}
Furthermore, if $\lambda=(1^{a_1},2^{a_2},\ldots,k^{a_k})$, then
\begin{equation}\label{3eq6}
R_v=\sum_{i=0}^{k-1}\frac{i!(k-1-i)!}{k}\sum_{<r_1,\ldots,r_i>}{a_1-1\choose r_1}{a_2\choose r_2}\cdots{a_i\choose r_i}(-1)^{r_2+r_4+r_6+\cdots},
\end{equation}
where $<r_1,\ldots,r_i>$ ranges over all non-negative integer solutions to the equation $\sum_j jr_j=i$.
\end{corollary}
\proof The number $R_v$ is equal to the number of different ways to
factorize $D_v$ into a permutation with one cycle (e.g., $s_v$) and
the other permutation with one cycle (e.g., $\pi_v^{-1}$).
Then, Eq.~\eqref{3eq5} follows from Corollary~\ref{2cor1}.
The explicit formula Eq.~\eqref{3eq6} follows from Stanley~\cite{stan}.
This completes the proof. \qed

\begin{example}
Given a plane permutation
\begin{eqnarray*}
\left(\begin{array}{cccccccccccccccccccc}
1&2&\cdots&7&8&9&10&11&12&13&14&15&16&17&18&19&20\\
1&7&\cdots&17&14&5&20&16&6&9&19&12&8&4&15&11&2
\end{array}\right),
\end{eqnarray*}
the corresponding unicellular map of which is shown on the left in Figure~\ref{fig-fig2}.
Consider the vertex
$
v=\Big(\begin{array}{ccccc}
8&11&14&16&19\\
14&16&19&8&11
\end{array}\Big),
$

\begin{eqnarray*}
D_v&=&\left(\begin{array}{ccccc}
8&19&16&14&11
\end{array}\right)\\
&=&\left(\begin{array}{ccccc}
8&11&14&16&19
\end{array}\right)\left(\begin{array}{ccccc}
8&16&11&19&14
\end{array}\right)\\
&=&\left(\begin{array}{ccccc}
8&14&19&11&16
\end{array}\right)\left(\begin{array}{ccccc}
8&14&19&11&16
\end{array}\right).
\end{eqnarray*}
Rearranging the half edges around the vertex $v$ following the second factorization of $D_v$, we
obtain another unicellular map as shown on the right hand side in Figure~\ref{fig-fig2}.

\begin{figure}[!htb]
\centering
\includegraphics[width=0.75\textwidth]{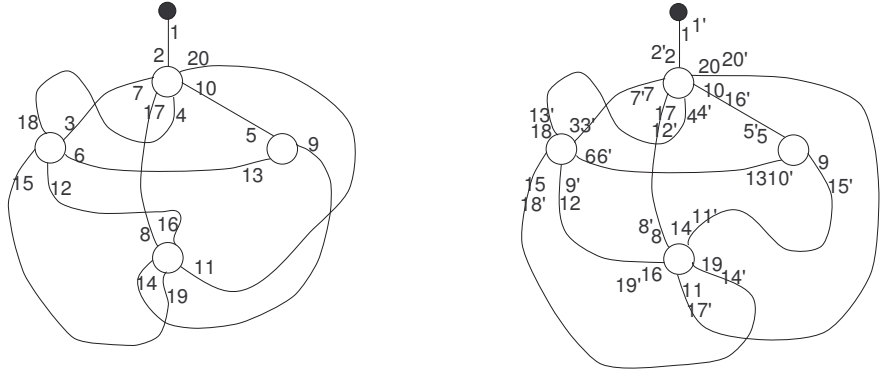}
\caption{A unicellular map with $10$ edges (left) and rearranging half
 edges around one of its vertices (right) where after relabeling
the boundary is $(1',2',\ldots,20')$.}\label{fig-fig2}

\end{figure}
\end{example}

We can see that given a one-face embedding of the graph $G$, randomly rearranging the half edges around the vertex
$v$, the probability of the resulting map to be unicellular is exactly $\frac{R_v}{(|H(v)|-1)!}$.
Furthermore, we have
\begin{theorem}
Let $\mathbb{E}$ be a one-face embedding of $G$, and $v$ is a vertex of $G$ with degree $deg(v)=k$.
Assume $v=(s_v,\pi_v)$ and $D_v$ is of cycle-type $\lambda=(1^{a_1},2^{a_2},\ldots,k^{a_k})$.
Then, the probability $prob_1(v)$ of the resulting embedding to be unicellular after rearranging the half edges
around $v$ satisfies
\begin{align}
\frac{2}{deg(v)-a_1+2} \leq prob_1(v) \leq \frac{2}{deg(v)-a_1+\frac{19}{29}}.
\end{align}
In particular, for any vertex $v$, $prob_1(v)\geq \frac{2}{deg(v)+2}$.
\end{theorem}
\proof In Zagier~\cite{zag}, it was proved that
$$
\frac{2(k-1)!}{k-a_1+2} \leq p_1^{\lambda}(k) \leq \frac{2(k-1)!}{k-a_1+\frac{19}{29}}.
$$
However, there are $(k-1)!$ different ways to arrange the half edges around $v$.
Then, according to Corollary~\ref{3cor1}, the probability $prob_1(v)$ of the resulting embedding
to be unicellular after rearranging the half edges around $v$ satisfies
\begin{align*}
\frac{2}{deg(v)-a_1+2} \leq prob_1(v) \leq \frac{2}{deg(v)-a_1+\frac{19}{29}}.
\end{align*}
This completes the proof of the former part.
Since it always holds that $\frac{2}{deg(v)-a_1+2} \geq \frac{2}{deg(v)+2}$, the
latter part follows. \qed

 \par
 Now we come to study the second question:
 given a cellular map and a vertex there, whether it is the unique way of arranging all half edges
around the vertex to achieve the genus in the present map, i.e., keep one-face?

\begin{theorem}\label{3thm2}
Let $\mathbb{E}$ be a one-face embedding of $G$, and $v$ is a vertex of $G$ with $deg(v)\geq 4$.
Then there is at least one another way to arrange the half edges
around the vertex $v$ such that the obtained embedding $\mathbb{E}'$ has
the same genus as $\mathbb{E}$.
\end{theorem}
\begin{proof} Assume $d\geq 4$ and
\begin{eqnarray*}
v=(s_v,\pi_v)=\left(\begin{array}{ccccc}
v_1&v_2&\cdots&v_{d-1}&v_d\\
v_{i,1}&v_{i,2}&\cdots&v_{i,d-1}&v_{i,d}
\end{array}\right),
\end{eqnarray*}
where $\pi_v=(v_1,V_2,\cdots, V_{d-1},V_d)$.
Firstly, from Lemma~\ref{2lem2} we know that if there is $V_l=v_p, V_m=v_q$ and $1<l<m\leq d, 1<p<q\leq d$ (i.e., Case~$3$),
then there is at least one another way to arrange all half edges around the vertex $v$ to keep the genus
and we are done.
If this is not true, then we must have $\pi_v=(v_1,v_d,v_{d-1},\ldots,v_2)$.
In this case, we have
\begin{eqnarray*}
v=(s_v,\pi_v)=\left(\begin{array}{ccccc}
v_1&v_2&\cdots&v_{d-1}&v_d\\
v_{d}&v_{1}&\cdots&v_{d-2}&v_{d-1}
\end{array}\right),
\end{eqnarray*}
so that
\begin{eqnarray*}
D_v=\left\{
\begin{array}{cc}
(v_1,v_3,\ldots,v_d,v_2,v_4,\ldots,v_{d-1}), &\quad d\in odd,\\
(1,3,\ldots,d-1)(2,4,\ldots,d), &\quad d\in even.
\end{array}
\right.
\end{eqnarray*}
Next, we only need to show that if $d\geq 4$ we have $R_v\geq 2$ in all cases.
Applying the formula to compute $R_v$, if $d\in odd$ we have
\begin{eqnarray*}
R_v=\frac{(d-1)!}{d}\sum_{i=0}^{d-1}(-1)^i{d-1\choose i}^{-1}=\frac{2(d-1)!}{d+1}.
\end{eqnarray*}
The simplification of the summation is from the following formula \cite{spru}
\begin{eqnarray*}
\sum_{i=0}^n (-1)^i{x\choose i}^{-1}=\frac{x+1}{x+2}(1+(-1)^n{x+1\choose n+1}^{-1}).
\end{eqnarray*}
It is not hard to see that $R_v\geq 2$ if $d\geq 4$.
Similarly, if $4|d$ and $d\geq 4$, we have
\begin{eqnarray*}
R_v&=&\sum_{i=0}^{\frac{d}{2}-1}(-1)^i\frac{i!(d-1-i)!}{d}+
\sum_{i=\frac{d}{2}}^{d-1}(-1)^i\frac{i!(d-1-i)!}{d}[(-1)^i+(-1)^{i-\frac{d}{2}}{2\choose 1}(-1)]\\
&=&\frac{2(d-1)!}{d+1}(1-{d\choose \frac{d}{2}}^{-1}).
\end{eqnarray*}
If $d\in even$ and $4\nmid d$, we have
\begin{eqnarray*}
R_v&=&\sum_{i=0}^{\frac{d}{2}-1}(-1)^i\frac{i!(d-1-i)!}{d}+
\sum_{i=\frac{d}{2}}^{d-1}(-1)^i\frac{i!(d-1-i)!}{d}[(-1)^i+(-1)^{i-\frac{d}{2}}{2\choose 1}]\\
&=&\frac{2(d-1)!}{d+1}(1+{d\choose \frac{d}{2}}^{-1}).
\end{eqnarray*}
In both cases, if $d\geq 4$, it is not hard to show $R_v\geq 2$
since both $\frac{2(d-1)!}{d+1}$ and $(1-{d\choose \frac{d}{2}}^{-1})$ are increasing functions of $d$.
Therefore, in all cases, if $d\geq 4$, then $R_v\geq 2$. This
completes the proof.
\end{proof}

Note if the number of half edges around a vertex is $1$ or $2$, it is trivial.
The only special case is when the number of half edges around a vertex is $3$.
For such a vertex, it may be the unique arrangement of half edges around it to achieve the genus of the present
map. Along the discussion, we actually have the following corollary

\begin{corollary}\label{3cor3}
Any even permutation on $[n]$ with $n\geq 4$ has at least two different factorizations into
two $n$-cycles.
\end{corollary}
\proof Since $D_v=s_v\circ \pi_v^{-1}$ and both $s_v$ as well as $\pi_v$ have only one cycle,
$D_v$ is an even permutation. The proof for Theorem~\ref{3thm2} just implies that $D_v$
has at least $2$ factorizations into two $n$-cycles. \qed

\subsection{Variation of the embedding around more vertices}
 Next, we slightly generalize above results by considering
 changing the local structure around more vertices of the underlying graph and their local embeddings.

 Firstly, we study rearrangement of half edges around $m\geq 1$ vertices simultaneously and independently,
 i.e., the underlying graph is not changed.
 Given a plane permutation $(s,\pi)$ and $m$ vertices $V_1,\ldots,V_m$ in $\pi$.
 Similar as the case of single vertex above,
 we can represent all these vertices by the plane permutation $V_{1-m}=(s_{1-m},\pi_{1-m})$,
 where $s_{1-m}$ is the subsequence obtained from $s$ by keeping only half
 edges in $V_1,\ldots,V_m$ and $\pi_{1-m}$ is the restriction of $\pi$ to these half edges.

 Denote $Dsh_{1-m}$ the number of different ways of simultaneous rearrangement of half edges around
 $V_i$,~$(1\leq i\leq m)$, respectively, and keep the unicellular property.

 \begin{theorem}\label{3thm3}
 Given a one-face embedding of $G$ and $m$ vertices $V_1,\ldots,V_m$ there, $Dsh_{1-m}$ is equal to the number of
 different ways to factor $D_{V_{1-m}}$ into $\gamma\sigma$, where $\gamma$ has one cycle while $\sigma$ has $m$ disjoint cycles and
 each cycle is on the set of half edges of $V_i$, respectively.
 \end{theorem}
 \begin{proof} Applying the same idea of diagonal blocks rearrangement as in the case of single vertex completes the proof.
 \end{proof}

 Now for a plane permutation $(s,\pi)$ and $m$ vertices $V_1,\ldots,V_m$ in $\pi$,
 if the half edges belonging to one of these vertices are allowed to attach to another vertex
 among these $m$ vertices, i.e., change the incident relation of these vertices and half edges around them,
 how many different ways to keep one-face?
 Assume the degree distribution of these $m$ vertices is encoded by the partition $\mu$.
 Let $Le(V_1,\ldots,V_m;\mu)$ denote the number of different variations (including both local incident relation
 and local embedding) of these vertices to preserve the degree distribution and preserve one-face.
 Note the degree of a single vertex may change, but as a whole the degree distribution will not change.
 Let $Le(V_1,\ldots,V_m)$ denote the number of different variations (including both local incident relation
 and local embedding) of these vertices to keep the number of vertices and keep one-face.
Then, we have

 \begin{theorem}\label{3thm4}
 Assume the cycle-type of $D_{V_{1-m}}$ is $\lambda$ and the total number of half edges
 around these $m$ vertices are $q$. Then we have
 \begin{eqnarray}
 Le(V_1,\ldots,V_m;\mu)&=& f_{\mu,\lambda}(q),\\
 Le(V_1,\ldots,V_m) &=& p_{m}^{\lambda}(q).
 \end{eqnarray}
 \end{theorem}
\begin{remark}
The method to study local variation of maps in this section can be easily
employed to study local variation of hypermaps.
\end{remark}

\section{Embeddings with $k$ faces and $k$-cyc plane permutations}
In this section, we generalize plane permutations $(s,\pi)$ to $k$-cyc plane permutations $(s,\pi)_k$ where
$s$ has $k$ cycles, in order to study graph embeddings with $k$ faces.

\begin{definition}
A $k$-cyc plane permutation on $[n]$ is a pair $\mathfrak{p}=(s,\pi)$ where $s$
is a permutation having $k$ cycles and $\pi$ is an arbitrary permutation. The permutation $D_{\mathfrak{p}}=
s\circ \pi^{-1}$ is called the diagonal of $\mathfrak{p}$.
\end{definition}\label{2def1}

Assume $s=(s_{11}, \ldots s_{1m_1})(s_{21},\ldots s_{2m_2})\cdots (s_{k1},\ldots s_{km_k})$,
where $\sum_i m_i =n$.
A $k$-cyc plane permutation $(s,\pi)_k$ can be represented by two aligned rows:
\begin{align*}
(s,\pi)_k=\left(\begin{array}{ccccccccccc}
\boxed{s_{11}} & s_{12} & \cdots & s_{1m_1} & \boxed{s_{21}}&\cdots &\quad s_{2m_2}\quad \cdots &\boxed{s_{k1}}& \cdots & s_{km_k}\\
\pi(s_{11})&\pi(s_{12})&\cdots & {\boxed{\pi(s_{1m_1)}}}& \pi(s_{21}) & \cdots  & \boxed{\pi(s_{2m_2})}
\cdots & \pi(s_{k1}) & \cdots & \boxed{\pi(s_{km_k})}
\end{array}\right)
\end{align*}
where each adjacent pair of ``boxed" elements (one is on the top row and the other is on the bottom row, e.g.,
$s_{11}$ and $\pi(s_{1m_1}$))
indicates a face.
Then, $D_{\mathfrak{p}}$ can be explicitly defined as follows:
\begin{itemize}
\item  For $1 \leq i \leq k$, $D_{\mathfrak{p}}(\pi(s_{ij}))=s_{i(j+1)}$ if $j\neq m_i$;
\item For $1 \leq i \leq k$, $D_{\mathfrak{p}}(\pi(s_{im_i}))=s_{i1}$.
\end{itemize}
Since every embedding with $k$ faces can be encoded into a triple $(\alpha, \beta, \gamma)$
where $\gamma=\alpha \beta$ and $\gamma$ has $k$ cycles, every embedding can be encoded into
a $k$-cyc plane permutation as well, i.e., $s=\gamma$, $\pi=\beta$ and $D_{\mathfrak{p}}=\alpha$.
A $k$-cyc plane permutation can be viewed as a concatenation of $k$ ``pseudo" plane permutations
induced by $k$ faces, where the ``pseudo" plane permutation induced by the face $f_i$ is
\begin{align*}
\left(\begin{array}{cccc}
\boxed{s_{i1}} & s_{i2} & \cdots & s_{im_i} \\
\pi(s_{i1})&\pi(s_{i2})&\cdots & {\boxed{\pi(s_{im_i)}}}
\end{array}\right).
\end{align*}
We will denote $f_i$ this ``pseudo" plane permutation if no confusions occur.
Let $H(f)$ denote the set of half edges contained in the face $f$.

\begin{lemma}\label{4lem1}
Let $v$ be a vertex of the graph $G$ and $\mathbb{E}$ be an embedding of $G$, where $v$ is incident
to $q$ faces, $f_i$, for $1\leq i \leq q$. Let $\mathbb{E}'$ be another embedding which
is obtained by rearranging the half edges around $v$ so that $v$ is incident to $q'$ faces,
$f'_i$, for $1 \leq i \leq q'$.
Then,
$$
\bigcup_{i=1}^q H(f_i)=\bigcup_{i=1}^{q'} H(f'_i), \quad\quad q\equiv q' \pmod{2}.
$$
\end{lemma}
\proof Firstly, every face $f$ of $\mathbb{E}$ can be expressed as $\{D_{\mathfrak{p}}\pi(z), (D_{\mathfrak{p}}\pi)^2(z), \ldots\}$
for any $z\in H(f)$.
If $f\neq f_i$ for $1\leq i \leq q$, then for any $z\in f$, $\pi'(z)=\pi(z)$ and $D_{\mathfrak{p}}(\pi(z))=D_{\mathfrak{p}}(\pi'(z))$.
Hence, the face $f$ in $\mathbb{E}$ is a face in $\mathbb{E}'$.
Next, we will show that $\bigcup_i H(f_i)$ will reorganize into $q'$ faces, $f'_i$, for $1 \leq i \leq q'$,
and all these $q'$ faces are incident to $v$.
It suffices to show each $f'_i$ contains at least one half edge of $v$.
For any half edge $u\in \bigcup_i H(f_i)$ that does not belong to $H(v)$,
assume $u$ is contained in the face $f_j$ of
$\mathbb{E}$ while in the face $f'_k$ of $\mathbb{E}'$.
Then,
\begin{align*}
f_j &=\{D_{\mathfrak{p}}\pi(u), (D_{\mathfrak{p}}\pi)^2(u), \ldots, v_i, D_{\mathfrak{p}}(\pi(v_i)),\ldots\}\\
f'_k &=\{D_{\mathfrak{p}}\pi(u), (D_{\mathfrak{p}}\pi)^2(u), \ldots\}
\end{align*}
where $v_i$ is the first half edge of $v$ appeared in $f_j$.
We know that, if $\pi'(z)=\pi(z)$ then $D_{\mathfrak{p}}(\pi(z))=D_{\mathfrak{p}}(\pi'(z))$.
Thus, $v_i \in H(f'_k)$ and the segment from $D_{\mathfrak{p}}(\pi(u))$ to $v_i$ in $f_j$ is
completely the same as the segment from $D_{\mathfrak{p}}(\pi(u))$ to $v_i$ in $f'_k$.
Therefore, there is no face among $f'_i$ for $1\leq i \leq q'$ which does not
contain a half edge of $v$.
Finally, the parity equivalence between $q$ and $q'$ comes from Euler characteristic formula.
This completes the proof. \qed

\begin{corollary}\label{4cor1}
Let $\mathbb{E}$ be an embedding of the graph $G$ and $v$ be a vertex of $G$, where $v$ is of degree $deg(v)$
and incident to $q$ faces in $\mathbb{E}$.
Assume $\mathbb{E}'$ is another embedding which is obtained by rearranging the half edges around $v$.
Then,
\begin{align}
-\lfloor \frac{deg(v)-q}{2}\rfloor \leq g(\mathbb{E}')-g(\mathbb{E})\leq \lfloor \frac{q-1}{2}\rfloor
\end{align}
\end{corollary}
\proof According to Lemma~\ref{4lem1}, rearranging half edges around $v$ will at most increase the number of faces
by $deg(v)-q$ and at most decrease the number faces by $q-1$ whence the corollary. \qed

Let $\mathbb{E}$ be an embedding of the graph $G$ and $v$ be a vertex of $G$.
Assume $v$ is incident to $q$ faces. Then, $v$ can be naturally encoded into a
$q$-cyc plane permutation $(s_v,\pi_v)_q$ obtained as follows: $s_v$ has $q$ cycles, where
each cycle is obtained by deleting all half edges in a face incident to $v$ except the elements
in $H(v)$, i.e., each cycle is induced from the cyclic order of a face incident to $v$,
$\pi_v$ is the restriction of $\pi$ to $H(v)$.

Given an embedding $\mathbb{E}$ and a vertex $v$, where $v$ is incident
to the faces $f_i$ for $1 \leq i \leq q$. Then, similar to one-face case,
the half edges in $H(v)$ will segment the $q$ ``pseudo" plane permutations corresponding to
these $q$ faces into diagonal blocks as well, where each diagonal block is determined by the left lower corner $\pi_v(x)$
and the right upper corner $s_v(x)$. We will show that rearranging the half edges around
$v$ is to rearrange these diagonal blocks.

\begin{theorem}\label{4thm1}

Let $\mathbb{E}$ be an embedding of the graph $G$ and $v$ be a vertex of $G$.
Assume $v$ is incident to $q$ faces. Let $D_v$ be the diagonal of $(s_v,\pi_v)_q$, and
denote $R_v(\Delta g)$ the number of different ways to arrange the half edges around $v$ such that
the obtained embedding $\mathbb{E}'$ has genus $g(\mathbb{E}')=g(\mathbb{E})+\Delta g$.
Then, we have
\begin{align}
R_v(\Delta g)= p^{\lambda(D_v)}_{q+2\Delta g}(deg(v)),
\end{align}
where $\lambda(D_v)$ is the cycle-type of $D_v$.
\end{theorem}
\proof We prove the theorem by showing that for every cyclic arrangement $\theta$ of the half edges
around $v$ such that the obtained embedding $\mathbb{E}'$ has genus $g(\mathbb{E}')=g(\mathbb{E})+\Delta g$
satisfies that $D_v\circ \theta$ has $q+2\Delta g$ cycles, and each $\theta$ that
$D_v\circ \theta$ has $q+2\Delta g$ cycles gives an embedding $\mathbb{E}'$ has
genus $g(\mathbb{E}')=g(\mathbb{E})+\Delta g$.

$(\Longrightarrow)$: Suppose $v$ is incident to $q$ faces in $\mathbb{E}$, $f_1,\ldots f_q$.
Thus, $H(v)= \bigcup_i H(f_i)$.
Assume $\theta$ on $H(v)$ gives an embedding $\mathbb{E}'$ with genus
$g(E)+\Delta g$.
According to Lemma~\ref{4lem1}, $H(v)$ will be reorganized into $q+ 2\Delta g$ faces, $f'_1,\ldots, f'_{q+2 \Delta g}$.
The rest of faces in $\mathbb{E}$ will not be impacted.
We will show that $D_v \circ \theta$ has $q+2\Delta g$ cycles, where each cycle is uniquely induced
from one face $f'_i$. Given the face $f'_i$,
\begin{align*}
f'_i= \left(\begin{array}{ccccccccccc}
\boxed{v'_{i1}} & x_1 & \cdots & v'_{i2} & x_2 & \cdots & v'_{i3} &\cdots & v'_{it_i} & \cdots & y\\
v'_{ij_1} & \cdots & x'_1 & v'_{ij_2} &\cdots & x'_2 & v'_{ij_3} & \cdots & v'_{ij_{t_i}} &\cdots & \boxed{z}
\end{array}\right),
\end{align*}
where $v'_{ik}, v'_{ij_k} \in H(v)$, $v'_{ij_k}=\theta(v'_{ik})$.
By the same reasoning as the proof for Theorem~\ref{3thm1}, the diagonal block
\begin{align*}
\begin{array}{cccc}
 & x_1 & \cdots & v'_{i2} \\
v'_{ij_1} & \cdots & x'_1 &
\end{array}
\end{align*}
is also a diagonal block in $\mathbb{E}$, which implies $D_v(v'_{ij_1})=v'_{i2}$.
Therefore, $D_v\circ \theta (v'_{i1})=v'_{i2}$. Considering all other diagonal blocks,
we have $(v'_{i1} v'_{i2} \cdots v'_{it_i})$ is a cycle of $D_v\circ \theta$.\\
$(\Longleftarrow)$: given any $\theta$ on $H(v)$ such that $D_v\circ \theta$ has $q+2\Delta g$ cycles,
it will induce a $(q+2\Delta g)$-cyc plane permutation $v=(s'_v, \pi'_v)_{q+2\Delta g}$,
\begin{align*}
v= \left(\begin{array}{ccccccccccc}
\boxed{v'_{11}} & \cdots & v'_{1t_1} & \boxed{v'_{21}} & \cdots & v'_{2t_2} & \cdots & \boxed{v'_{(q+2\Delta g)1}} &
\cdots & v'_{(q+2\Delta g)t_{q+2\Delta g}}\\
\theta({v'_{11}}) & \cdots & \boxed{\theta(v'_{1t_1})} & \theta(v'_{21}) & \cdots & \boxed{\theta(v'_{2t_2})} &
\cdots & {\theta(v'_{(q+2\Delta g)1})} & \cdots & \boxed{\theta(v'_{(q+2\Delta g)t_{q+2\Delta g}})}
\end{array}\right)
\end{align*}
By extending the pair $(\theta(v_{ik}),s'_v(v_{ik}))$ into the diagonal block which has $\theta(v_{ik})$
as the left lower corner and $s'_v(v_{ik})$ as the right upper corner, and
concatenating with the rest of faces in $\mathbb{E}$, we obtain
an embedding $\mathbb{E}'$ with $2\Delta g$ more faces than $\mathbb{E}$,
i.e., $g(\mathbb{E}')=g(\mathbb{E})+\Delta g$.
It can be shown that only the half edges around the vertex $v$ are cyclically arranged in different manners
in $\mathbb{E}$ and $\mathbb{E}'$, so each $\theta$ uniquely induces an embedding
$\mathbb{E}'$ which is obtained by rearranging the half edges around $v$ and
has genus $g(\mathbb{E}')=g(\mathbb{E})+\Delta g$.
This completes the proof. \qed

In fact, by similar reasoning, we can obtain a more general result.
Given an embedding $\mathbb{E}$ of the graph $G$ and one vertex $v$ of $G$,
let $deg(v)=k$ and assume there are $a_i$ faces of $\mathbb{E}$ where each contains $i$ half edges in $H(v)$.
We call $\mu=1^{a_1}2^{a_2}\cdots k^{a_k}$ the f-incidence degree distribution of $v$ w.r.t. $\mathbb{E}$.
Now how many embeddings $\mathbb{E}'$ where the f-incidence degree distribution of $v$ is $\eta$ can
be obtained by rearranging the half edges around $v$? we denote this number as $R_v(\eta)$.

\begin{theorem}\label{4thm2}
\begin{align}
R_v(\eta)= f_{\eta,\lambda(D_v)}(deg(v)).
\end{align}
\end{theorem}

As the first corollary of Theorem~\ref{4thm1}, we obtain the following (local) version of ``interpolation" theorem.
\begin{corollary}\label{4cor2}
Let $\mathbb{E}$ be an embedding of the graph $G$. If there is a vertex $v=(s_v, \pi_v)_q$,
then there exists an embedding $\mathbb{E}'$ of $G$ such that
$g(\mathbb{E}')=g(\mathbb{E})+\Delta g$ for any
$$
-\lfloor\frac{deg(v)+1-\ell(\lambda(D_v))-q}{2}\rfloor \leq \Delta g \leq \lfloor \frac{q-1}{2}\rfloor.
$$
In particular, if there is a vertex $v$ of $G$ which is incident to every face of $\mathbb{E}$,
then $G$ is upper embeddable.
\end{corollary}
\proof According to Corollary~\ref{2cor1}, $p^{\lambda}_k(n)\neq 0$ as long as $p^{\lambda}_{k+2i}(n)\neq 0$ for some $i>0$.
And from Proposition~\ref{2pro1}, $p^{\lambda}_{deg(v)+1-\ell(\lambda(D_v))}(deg(v))\neq 0$.
Therefore, for any
$$
-\lfloor\frac{deg(v)+1-\ell(\lambda(D_v))-q}{2}\rfloor \leq \Delta g \leq \lfloor \frac{q-1}{2}\rfloor,
$$
$p^{\lambda(D_v)}_{q-2\Delta g}(deg(v))\neq 0$. Namely, rearranging the half edges around
$v$ can lead to an embedding $\mathbb{E}'$ such that $\mathbb{E}'$ has $2\Delta g$ less faces.
Hence, $g(\mathbb{E}')=g(\mathbb{E})+ \Delta g$.
If there is a vertex $v$ of $G$ which is incident to every face of $\mathbb{E}$,
then rearranging the half edges around $v$ can lead to an embedding with either $1$ face
or $2$ faces, depending on the parity of the number of faces in $\mathbb{E}$.
Thus, $G$ is upper embeddable. \qed

In fact, the result in Corollary~\ref{4cor2} can be further optimized. Given two vertices,
if there is no face of the embedding $\mathbb{E}$ incident to both vertices,
or if there is only one face $f_0$ of the embedding $\mathbb{E}$ incident to both vertices
where all the half edges of one vertex contained
in $f_0$ are completely contained in a diagonal block determined by the other vertex,
the two vertices are called $\mathbb{E}$-facial disjoint.
Applying Lemma~\ref{4lem1} and diagonal blocks rearrangement argument,
if two vertices are $\mathbb{E}$-facial disjoint,
re-embedding them simultaneously will not interfere with each other.
Hence, we have

\begin{corollary}\label{4cor22}
Let $\mathbb{E}$ be an embedding of the graph $G$. If vertices $v_i=(s_{v_i}, \pi_{v_i})_{q_i}$, $1\leq i \leq m$,
are mutually $\mathbb{E}$-facial disjoint,
then there exists an embedding $\mathbb{E}'$ of $G$ such that
$g(\mathbb{E}')=g(\mathbb{E})+\Delta g$ for any
$$
\sum_{i=1}^m -\lfloor\frac{deg(v_i)+1-\ell(\lambda(D_{v_i}))-q_i}{2}\rfloor \leq \Delta g \leq \sum_{i=1}^m \lfloor \frac{q_i-1}{2}\rfloor.
$$
\end{corollary}
\proof Since $v_i$ are mutually $\mathbb{E}$-facial disjoint, the range of genus difference achieved by
re-embedding $v_i$ do not interfere with each other. Therefore, the differences can be combined together
whence the corollary. \qed

\begin{corollary}\label{4cor3}
Let $\mathbb{E}$ be an embedding of the graph $G$ with genus $g_{max}(G)$.
Then, every vertex is incident to at most $2$ faces.
\end{corollary}
\proof If there is a vertex $v$ with no less than $3$ faces, according to Corollary~\ref{2cor1},
there exists an embedding with $2$ less faces. Hence, $\mathbb{E}$ can not be
an embedding of the graph $G$ with the maximum genus $g_{max}(G)$.\qed

The fact that if there exists a vertex incident to at least $3$ faces in an embedding, an embedding
with higher genus always exists has been well known in the literature, e.g., in~\cite{xuong2,martin}.
In particular, Corollary~\ref{4cor3} is the same as a very recent result in~\cite{martin} where
locally maximal embedding
is studied. In our context, a locally maximal embedding can be defined as an embedding from which
a higher genus embedding can not be obtained by rearranging the half edges around one of the vertices.
Then, we can actually restate that if an embedding is locally maximal,
then every vertex is incident to at most $2$ faces.

However, to the best of our knowledge, it seems there is no simple characterization to determine
if there exists a lower genus embedding based on a given embedding. Thus, the lower bound of $\Delta g$
in Corollary~\ref{4cor2} may be the first simple characterization. Furthermore, we obtain the
following necessary condition for an embedding of $G$ to be an embedding with the minimum genus.
We mention that there is a sufficient condition for an embedding of $G$ to be an embedding with the minimum genus
in Thomassen~\cite{thoma2}.

\begin{corollary}\label{4cor4}
Let $\mathbb{E}$ be an embedding of the graph $G$ with genus $g_{min}(G)$,
and a vertex $v=(s_v, \pi_v)_{q_v}$. Then,
\begin{align}
\ell(\lambda(D_v))+q_v=deg(v)+1.
\end{align}
\end{corollary}
\proof Otherwise, we can increase the number of faces by rearranging the half edges
around $v$ so that we obtain an embedding with even lower genus, which
contradict the fact that $\mathbb{E}$ is an embedding of the graph $G$ with genus $g_{min}(G)$. \qed

Surely, we can define locally minimal embedding analogously. So, we obtain that
if an embedding is locally minimal, then, for every vertex $v$,
\begin{align}
\ell(\lambda(D_v))+q_v=deg(v)+1.
\end{align}
It is also obvious that if the embedding $\mathbb{E}$ is locally minimal and all vertices of $G$ are mutually
$\mathbb{E}$-facial disjoint, then the genus of $\mathbb{E}$ is equal to $g_{min}(G)$.

In addition, Corollary~\ref{4cor2} provides an easy approach to estimate the range $[g_{min}(G), g_{max}(G)]$
for a given graph $G$: we can randomly try several embeddings of $G$, and from each embedding, we obtain an estimate
for $[g_{min}(G), g_{max}(G)]$ using the following theorem, and finally combine these estimates.

\begin{theorem}\label{4thm3}
Let $\mathbb{E}$ be an embedding of the graph $G$ and for a vertex $v$ of $G$, assume $v=(s_v, \pi_v)_{q_v}$.
Let
$$
T_1=\min_v \{-\lfloor\frac{deg(v)+1-\ell(\lambda(D_v))-q_v}{2}\rfloor\}, \quad T_2=\max_v \{\lfloor \frac{q_v-1}{2}\rfloor\}.
$$
Then, we have
\begin{align}
g_{min}(G)\leq g(\mathbb{E})+T_1 \leq g(\mathbb{E})+T_2 \leq g_{max}(G).
\end{align}
\end{theorem}

We remark that this approach to estimate the genus range can be optimized in the
similar manner as in Corollary~\ref{4cor22}, i.e., the minimum (resp. maximum) can
take over all $\mathbb{E}$-facial disjoint covers, where a $\mathbb{E}$-facial disjoint cover
is a set of mutually $\mathbb{E}$-facial disjoint vertices such that the union of
their incident faces is the set of all faces of $\mathbb{E}$.
Additionally, it may be possible to obtain a more efficient procedure to determine $[g_{min}(G),g_{max}(G)]$
if we combine our approach here and other algorithms to generate an embedding (if there is) of $G$ on a given 
surface of genus $g$, e.g., the linear algorithm in~\cite{kmr}. A rough idea could be as follows: suppose we know that 
$0\leq g_{min}(G)\leq a \leq b \leq g_{max}(G)\leq \lfloor\frac{\beta(G)}{2}\rfloor$. We can next choose a number
in $[0,a)$ or $(b,\lfloor\frac{\beta(G)}{2}\rfloor]$, say $k$. Then, apply the linear algorithm to 
generate an embedding of $G$ on $S_k$ and extend the range $[a,b]$ based on the obtained embedding by our approach. 
If there is no embedding on $S_k$, we can update the outer bound, i.e., $[0,\lfloor\frac{\beta(G)}{2}\rfloor]$.
Iterating this procedure, we can eventually obtain $[g_{min}(G),g_{max}(G)]$.

In the following, we present an analogue of Case~$3$ (and Case~$4,5,6$), Case~$1$ and Case~$2$ in Lemma~\ref{2lem2},
which increases the genus by $0$,~$1$ and $-1$, respectively.
In the derivation, a kind of local Poincar\'{e} dual is applied.

\begin{proposition}\label{4pro1}
Let $\mathbb{E}$ be an embedding of the graph $G$ and a vertex $v=(s_v,\pi_v)_q$, where
$$
\pi_v=(s_{i-1},v_1^i,\ldots v_{m_i}^i, s_j, v_1^j, \ldots v_{m_j}^j, s_l, v_1^l,\ldots v_{m_l}^l).
$$
If in $\mathbb{E}$, there exists a face of the form
$
(s_{i-1},\ldots s_j, \ldots s_l, \ldots),
$
or two faces of the form
$$
(s_{i-1},\ldots s_j, \ldots)(s_l, \ldots),
$$
then rearranging $H(v)$ according to the cyclic order
$$
(s_{i-1},v_1^j,\ldots v_{m_j}^j, s_l, v_1^i,\ldots v_{m_i}^i, s_j, v_1^l, \ldots v_{m_l}^l)
$$
will lead to the embedding $\mathbb{E}'$ with $g(\mathbb{E}')=g(\mathbb{E})$.
\end{proposition}

\proof Since $v=(s_v, \pi_v)_q$, we have $s_v=D_v \circ \pi_v$, where $s_v$ has
$q$ cycles and $\pi_v$ has only one cycle. This is equivalent to
$\pi_v =D_v^{-1} \circ s_v$ which corresponds to a plane permutation $(\pi_v, s_v)$ with
diagonal $D_v^{-1}$, i.e., a kind of local Poincar\'{e} dual.
Now the given conditions in the proposition either agree with
Case~$3$ or one of $\{\mbox{Case~$4$},\mbox{Case~$5$}, \mbox{Case~$6$} \}$ in Lemma~\ref{2lem2}.
Namely, if we transpose $\pi_v$ into
$$
(s_{i-1},v_1^j,\ldots v_{m_j}^j, s_l, v_1^i,\ldots v_{m_i}^i, s_j, v_1^l, \ldots v_{m_l}^l),
$$
we obtain a new plane permutation $(\pi'_v, s'_v)$ where the number of cycles in $s'_v$ equals
to the number of cycles in $s_v$. That is, rearranging $H(v)$ according to the cyclic order
$$
(s_{i-1},v_1^j,\ldots v_{m_j}^j, s_l, v_1^i,\ldots v_{m_i}^i, s_j, v_1^l, \ldots v_{m_l}^l)
$$
will not change the number of faces of the embedding. Therefore,
the resulting embedding $\mathbb{E}'$ satisfies $g(\mathbb{E}')=g(\mathbb{E})$.\qed

\begin{proposition}\label{4pro2}
Let $\mathbb{E}$ be an embedding of the graph $G$ and a vertex $v=(s_v,\pi_v)_q$, where
$$
\pi_v=(s_{i-1},v_1^i,\ldots v_{m_i}^i, s_j, v_1^j, \ldots v_{m_j}^j, s_l, v_1^l,\ldots v_{m_l}^l).
$$
If $s_{i-1}$,~$s_j$ and $s_l$ are contained respectively in three faces in $\mathbb{E}$,
then rearranging $H(v)$ according to the cyclic order
$$
(s_{i-1},v_1^j,\ldots v_{m_j}^j, s_l, v_1^i,\ldots v_{m_i}^i, s_j, v_1^l, \ldots v_{m_l}^l)
$$
will lead to the embedding $\mathbb{E}'$ with $g(\mathbb{E}')=g(\mathbb{E})+1$.
\end{proposition}
\proof After applying the ``local Poincar\'{e} dual", the given conditions in the proposition agree with
Case~$1$ in Lemma~\ref{2lem2} whence the proposition. \qed

\begin{proposition}\label{4pro3}
Let $\mathbb{E}$ be an embedding of the graph $G$ and a vertex $v=(s_v,\pi_v)_q$, where
$$
\pi_v=(s_{i-1},v_1^i,\ldots v_{m_i}^i, s_j, v_1^j, \ldots v_{m_j}^j, s_l, v_1^l,\ldots v_{m_l}^l).
$$
If in $\mathbb{E}$, there exists a face of the form
$
(s_{i-1},\ldots s_l, \ldots s_j, \ldots),
$
then rearranging $H(v)$ according to the cyclic order
$$
(s_{i-1},v_1^j,\ldots v_{m_j}^j, s_l, v_1^i,\ldots v_{m_i}^i, s_j, v_1^l, \ldots v_{m_l}^l)
$$
will lead to the embedding $\mathbb{E}'$ with $g(\mathbb{E}')=g(\mathbb{E})-1$.
\end{proposition}
\proof After applying the ``local Poincar\'{e} dual", the given conditions in the proposition agree with
Case~$2$ in Lemma~\ref{2lem2} whence the proposition. \qed

\subsection*{Acknowledgements}

We acknowledge the financial support of the Future and Emerging
Technologies (FET) programme within the Seventh Framework Programme (FP7) for
Research of the European Commission, under the FET-Proactive grant agreement
TOPDRIM, number FP7-ICT-318121.



\begin{thebibliography}{99}

\bibitem{reidys} J.E. Andersen, R.C. Penner, C.M. Reidys, M.S. Waterman, Topological classification and enumeration of RNA structures by genus,
\bibitem{chap} G. Chapuy, A new combinatorial identity for unicellular maps, via a direct bijective approach,  Adv. in Appl. Math, 47:4 (2011), 874-893.J. Math. Biol. 2013 Nov;67(5):1261-78

\bibitem{chr-1} R.X.F. Chen, C.M. Reidys, Another combinatorial proof of a result of Zagier and Stanley, arXiv:1502.07674 [math.CO].
\bibitem{chr-2} R.X.F. Chen, C.M. Reidys, A simple framework on sorting permutations, arXiv:1502.07971 [math.CO].
\bibitem{chr-3} R.X.F. Chen, C.M. Reidys, Narayana polynomials and some generalizations, arXiv:1411.2530v2 [math.CO].

\bibitem{duke} A. Duke, The genus, regional number, and Betti number of a graph, Canad. J. Math. 18 (1966), 817-822.
\bibitem{edmonds} J. Edmonds, A Combinatorial Representation for Polyhedral Surfaces, Notices Amer. Math. Soc., vol. 7, (1960) A646.

\bibitem{furst} M.L. Furst, J.L. Gross, L.A. MeGeoch, Find a maximum genus graph imbedding, J. Assoc. Comput. Math., 35 (1988), 253-534.
\bibitem{gross} J.L. Gross, T.W. Tucker, Stratified graphs for imbedding systems,
Discrete Mathematics - DM, vol. 143, no. 1-3, 71-85, 1995.
\bibitem{IJ} I.P. Goulden, A. Nica, A direct bijection for the Harer-Zagier formula, J. Combin. Theory Ser. A, 111(2):224-238, 2005.
\bibitem{ag} A. Goupil, G. Schaeffer, Factoring $n$-cycles and counting maps of given genus, European J. Combin., 19(7):819-834, 1998.
\bibitem{hz} J. Harer, D. Zagier, The Euler characteristics of the moduli space of curves, Invent. Math., 85(3): 457-485, 1986.

\bibitem{liu2} Y. Huang, Y. Liu, Face size and the maximum genus of a graph 1: Simple graphs, J. Combin. Theory Ser. B 80 (2000), 356-370.

\bibitem{jac} D.M. Jackson, Counting cycles in permutations by group characters, with an application to a topological
problem, Trans. Amer. Math. Soc., 299(2):785-801, 1987.
\bibitem{jung} M. Jungerman, A characterization of upper-embeddable graphs, TIans. Amer. Math. Soc. 241 (1978), 401-406.
\bibitem{martin} M. Kotrb\v{c}\'{i}k, Martin \v{S}koviera, Locally-maximal embeddings of graphs in orientable surfaces,
The Seventh European Conference on Combinatorics, Graph Theory and Applications
CRM Series Volume 16, 2013, pp 215-220.
\bibitem{kmr} K. Kawarabayashi, B. Mohar, B.A. Reed, A Simpler Linear Time Algorithm for Embedding Graphs into an Arbitrary Surface 
and the Genus of Graphs of Bounded Tree-Width, FOCS 2008: 771-780.
\bibitem{liu} Y.P. Liu, The maximum non-orientable genus of a graph (in Chinese), Scientia Sinica (Special Issue on Math),I(1979),191-201.
\bibitem{lzv} S.K. Lando, A.K. Zvonkin, Graphs on surfaces and their applications, Encyclopaedia Math. Sci. 141, Springer- Verlag, Berlin, 2004.
\bibitem{nebe} L. Nebesky, A new characterizations of the maximum genus of graphs, Czechoslovak Math. J., 31 (106) (1981), 604-613.

\bibitem{white} E. Nordhaus, B. Stewart, A. White, On the maximum genus of a graph, J. Combin. Theory B, 11 (1971), 258-267.
\bibitem{pen} R.C. Penner, M. Knudsen, C. Wiuf, J.E. Andersen, Fatgraph models of proteins,
Comm. Pure Appl. Math., 63 1249-1297, 2010.
\bibitem{youngs} G. Ringel, J.W.T. Youngs, Solution of the Heawood map-coloring problem, Proc. Nat. Acad. Sci. U.S.A. 60 (1968), 438-445.
\bibitem{spru} R. Sprugnoli, Riordan array proofs of identities in Gould's book, 2006.
\bibitem{stan} P.R. Stanley, Factorization of permutation into $n$-cycles, Discrete Math. 37 (1981), 255-262.
\bibitem{thoma} C. Thomassen, The graph genus problem is NP-complete, J. Algorithms 10 (1989) 568-576.
\bibitem{thoma2} C. Thomassen, Embedding of graphs with no short non-contractible cycles, J. Comb. Theory, Ser. B 48 (1990), 155-177.
\bibitem{liu3} L. Wan, Y. Liu, Orientable embedding genus distribution for certain types of graphs,
J. Comb. Theory, Ser. B 98(1): 19-32 (2008).


\bibitem{walsh1} T.R.S. Walsh, A.B. Lehman, Counting rooted maps by genus I, J. Combinatorial Theory Ser. B13 (1972) 192-218.


\bibitem{xuong2} N.H. Xuong, How to determine the maximum genus of a graph, J. Combin. Theory Ser. B 26 (1979), 217-225.
\bibitem{xuong1} N.H. Xuong, Upper embeddable graphs and related topics, J. Combin. Theory B, 26 (1979), 226-232.








\bibitem{zv} A. Zvonkin, Matrix Integrals and Map Enumeration: An Accessible Introduction, Mathematical and Computer Modelling, Vol. 26 (1997).







\bibitem{zag} D. Zagier, On the distribution of the number of cycles of elements in symmetric groups,
 Nieuw Arch. Wisk. (4),13, No. 3 (1995), 489-495.

\end{thebibliography}
\end{document}